\documentclass{amsart}

\usepackage{cite}

\usepackage{amsmath}
\usepackage{amstext}
\usepackage{amssymb}

\usepackage{amsthm}

\theoremstyle{plain}
\newtheorem{theorem}{Theorem}[section]
\newtheorem{lemma}[theorem]{Lemma}
\newtheorem{corollary}[theorem]{Corollary}

\newtheorem{refthm}{Theorem}
\newtheorem*{ryabykhthm}{Ryabykh's Theorem}
\newtheorem*{cauchygreenthm}{Cauchy-Green Theorem}

\theoremstyle{definition}

\numberwithin{equation}{section}

\newcommand{\conj}[1]{\overline{#1}}
\DeclareMathOperator{\sgn}{sgn}

\DeclareMathOperator{\Rp}{Re}

\begin{document}


\title[An Extension of Ryabykh's Theorem For $1<p<\infty$]{Extremal Problems in Bergman Spaces and an Extension of 
Ryabykh's $H^p$ Regularity Theorem For $1<p<\infty$}
\author{Timothy Ferguson}
\address{Department of Mathematics\\University of Alabama\\Tuscaloosa, AL}
\email{tjferguson1@ua.edu}

\date{\today}

\begin{abstract}
We study linear extremal problems in the Bergman space 
$A^p$ of the unit disc, where $1 < p < \infty$. 
Given a functional on the dual space of $A^p$ with 
representing kernel $k \in A^q$, where $1/p + 1/q = 1$, 
we show that if $q \le q_1 < \infty$ and $k \in H^{q_1}$, then 
$F \in H^{(p-1)q_1}$. This result was previously 
known only in the case where $p$ is an even integer. 
We also discuss related results. 
\end{abstract}

\maketitle

An analytic function $f$ in the unit disc ${\mathbb{D}}$ belongs to the 
Bergman space $A^p$ if 
\begin{equation*}  \|f\|_{A^p} = \left\{ \int_{{\mathbb{D}}} |f(z)|^p d\sigma(z)\right\}^{1/p} < \infty, \end{equation*}
where $\sigma$ is normalized area measure
(so that $\sigma({\mathbb{D}})=1$). 
For $1<p<\infty$, each functional $\phi \in (A^p)^*$ can be uniquely 
represented by 
\begin{equation*}
\phi(f) = \int_{{\mathbb{D}}} f \conj{k}\, d\sigma
\end{equation*}
for some $k \in A^q$ (called the kernel of $\phi$), 
where $q = p/(p-1)$ is the conjugate index.

In this paper we study regularity results for 
the extremal problem of maximizing 
$\Rp\phi(f)$ among all functions $f \in A^p$ of unit norm.   
An important regularity result is Ryabykh's theorem, 
which states that if the kernel is 
actually in the Hardy space $H^q$, then the extremal function must be in 
the Hardy space $H^p$ (see \cite{Ryabykh} or \cite{tjf1} for a proof).  In 
\cite{tjf2}, the following extensions of Ryabykh's theorem are shown 
in the case where $p$ is an even integer:
\begin{itemize}
\item For $q \le q_1 < \infty$, the extremal function 
$F \in H^{(p-1)q_1}$ if the kernel $k \in H^{q_1}$ 
(if $q_1 = q$ this is Ryabykh's theorem).
\item If the Taylor coefficients of $k$ satisfy a certain bound, then 
$F \in H^\infty$.
\item The map sending a kernel $k \in H^q$ to 
its extremal function $F\in A^p$ is a continuous map from 
$H^q \setminus \{0\}$ into $H^p$. 
\item For $q \le q_1 < \infty$, if the extremal function 
$F \in H^{(p-1)q_1}$, then the kernel $k \in H^{q_1}$.
(In fact, the proof in \cite{tjf2} shows that this result holds 
if $1 < q_1 < \infty$). 
\end{itemize}
We show that the first two results above hold for 
all $p$ such that $1 < p < \infty$.  
We also show a weaker form of the third result holds for 
$1 < p < \infty$, while a weaker form of the fourth holds if 
$2 \le p < \infty$. 
It is an open problem 
whether the last two results hold in their strong forms for 
$1 < p < \infty$. 

To overcome certain technical difficulties in the proof, we rely on 
regularity results from \cite{Khavinson_Stessin} 
for extremal functions with polynomial kernels.  
These results rely on regularity 
theorems for complex analogues of $p$-harmonic functions.  Our paper 
also uses an inequality based on Littlewood-Paley theory that was 
proved in \cite{tjf2}.

\section{Extremal Problems and Ryabykh's Theorem}\label{intro}
We now introduce the topic of the paper in more detail.  
(See \cite{tjf2} for a slightly more detailed introduction). 
If $f$ is an 
analytic function, $S_n f$ denotes its $n^{th}$ Taylor polynomial 
at the origin. We denote Lebesgue area measure by $dA$, and 
normalized area measure by $d\sigma$, so that $\sigma(\mathbb{D}) = 1$.

We recall some basic facts about Hardy and Bergman spaces.  For proofs 
and further information, see \cite{D_Hp} and \cite{D_Ap}. 
Suppose that $f$ is analytic in the unit disc.  For $0 < p < \infty$ and 
$0 < r < 1,$ the $p^{\mathrm{th}}$ integral mean of $f$ 
at radius $r$ is 
\[M_p(f,r)=
\bigg\{ \frac{1}{2\pi} \int_{0}^{2\pi} |f(re^{i\theta})|^p d\theta 
\bigg\}^{1/p},\]
whereas if $p=\infty$ it is
\[
M_\infty(f,r) = \max_{0\le \theta < 2\pi} |f(re^{i\theta})|.\]
The integral means are increasing 
functions of $r$ for fixed $f$ and $p$. 
An analytic function $f$ is in 
the Hardy space $H^p$ if $M_p(f,r)$ is bounded.  
The radial limit 
$f(e^{i\theta}) = \lim_{r \rightarrow 1^-} f(re^{i\theta})$ 
exists for almost every $\theta$ if $f$ is an $H^p$ function.
For $0 < p < \infty$, 
we have that $f(re^{i\theta})$ approaches 
the boundary function $f(e^{i\theta})$ in $L^p(d \theta)$ 
as $r \rightarrow 1^{-}$. 
Two $H^p$ functions whose boundary values agree on some set of 
positive measure are identical.  
The space $H^p$ is a Banach space with norm 
\[\| f \|_{H^p} = \sup_r M_p(f,r) = \|f(e^{i\theta})\|_{L^p}.\]
Thus we can regard $H^p$ as a subspace of 
$L^p({\mathbb{T}}),$ where ${\mathbb{T}}$ denotes the unit circle. 
If $1 < p < \infty$, the space $H^p$ is reflexive. 

If $f\in H^p$ and $1<p<\infty$, then $S_n f \rightarrow f$ in 
$H^p$ as $n \rightarrow \infty$, where $S_n f$ is the 
$n^{\textrm{th}}$ partial sum of the Taylor series for $f$ centered at the 
origin.  
The Szeg\H{o} projection $S$ maps each function $f \in L^1({\mathbb{T}})$ to 
an analytic function defined by 
\begin{equation*}
Sf(z) = \frac{1}{2\pi}\int_0^{2\pi} \frac{f(e^{it})}{1-e^{-it}z} dt
\end{equation*}
for $|z| < 1$.
It fixes $H^1$ functions and maps $L^p$ boundedly onto $H^p$ for 
$1 < p < \infty$.   
If $f \in L^p$ for $1<p<\infty$ and 
$\displaystyle f(\theta) = \sum_{n=-\infty}^\infty a_n e^{in\theta},$ 
then $\displaystyle Sf(z) = \sum_{n=0}^\infty a_n z^n.$

For $1 < p < \infty,$ the dual space 
$(A^p)^*$ is isomorphic to $A^q$, 
where $1/p + 1/q = 1$.  A functional $\phi \in (A^p)^*$ corresponds
to $k\in A^q$ if $\phi(f) = \int_{{\mathbb{D}}} f(z)\conj{k(z)}\,d\sigma(z)$. 
This correspondence is conjugate linear and 
does not preserve norms, but it is the case that 
\begin{equation}\label{A_q_isomorphism}
\| \phi \|_{(A^p)^*} \le \| k \|_{A^q} \le C_p \| \phi \|_{(A^p)^*},
\end{equation}
where $C_p$ is a constant depending only on $p$.  It can be shown that 
$C_p \le \pi \csc(\pi/p)$ (see \cite{Dostanic_BP_Norm} and the proof 
of Theorem 6 in Section 2.4 of \cite{D_Ap}). 
As with Hardy spaces, 
if $f \in A^p$ for $1<p<\infty$, then $S_n f \rightarrow f$ in $A^p$ 
as $n\rightarrow \infty$.

In this paper the only Bergman spaces we consider are those with
$1<p<\infty$. 
For a given linear 
functional $\phi \in (A^p)^*$ such that $\phi \ne 0$,
we study the extremal problem of finding a function $F \in A^p$ 
with norm 
$\|F\|_{A^p} = 1$ such that 
\begin{equation}\label{norm1}
\Rp \phi(F) = \sup_{\|g\|_{A^p}=1} \Rp \phi(g) = \| \phi \|.
\end{equation}
Such a function $F$ is called an extremal function, and 
we say that $F$ is an extremal function for a function $k \in A^q$ 
if $F$ solves problem 
\eqref{norm1} for the functional $\phi$ with kernel $k$. 
Note that for $p=2$ the extremal function is $F = k/\|k\|_{A^2}.$ 

For $1 < p < \infty$ an extremal function always exists and is unique, 
which follows from 
the uniform convexity of $A^p$.  Also, for any function $F$ of 
unit $A^p$ norm, there is some $k$ such that $F$ solves \eqref{norm1} 
for the functional $\phi$ 
with kernel $k$, and such a $k$ is unique up to a positive scalar 
multiple. Furthermore, one such $k$ is given by 
$\mathcal{P}(|F|^p/ \overline{F})$, where $\mathcal{P}$ is the 
Bergman projection (see \cite{tjf1} and \cite{tjf3}).

This problem has been studied by many authors, e.g.\ in 
\cite{DKSS_Pac}, \cite{tjf3}, \cite{Hedenmalm_canonical_A2}, 
\cite{Khavinson_Stessin}, \cite{Ryabykh_certain_extp}
and \cite{Dragan}.
Regularity results for solutions to this and similar problems can 
be found in 
\cite{tjf1}, \cite{tjf2},
\cite{Khavinson_McCarthy_Shapiro} and
\cite{Ryabykh}.
See also the survey \cite{Beneteau_Khavinson_survey}.

Even though it is well known, we restate 
the Cauchy-Green theorem, which is an important tool in this paper.
\begin{cauchygreenthm}
If $\Omega$ is a region in the plane with 
piecewise smooth boundary and $f\in C^1({\overline{\Omega}})$, then
\[ \frac{1}{2 i} \int_{\partial \Omega} f(z) \, dz = \int_{\Omega} 
\frac{\partial}{\partial \conj{z}} f(z) \,dA(z),\]
where $\partial \Omega$ denotes the 
boundary of $\Omega$.
\end{cauchygreenthm}

The next result is an important characterization of extremal functions in 
$A^p$ for $1<p<\infty$ (see \cite{Shapiro_Approx}, p.~55). The last part 
of the theorem follows from the previous parts by a standard 
approximation argument. 
\begin{refthm}\label{integral_extremal_condition} 
Let $1 < p < \infty$ and let $\phi \in (A^p)^*$.   
A function $F \in A^p$ with $\|F\|_{A^p} = 1$ satisfies 
\[\Rp \phi(F) = \sup_{\|g\|_{A^p} =1} \Rp \phi(g) = \| \phi \|\]
if and only if $\Rp \phi(F) > 0$ and 
\[\int_{{\mathbb{D}}} h |F|^{p-1} \conj{\sgn F}  \, d\sigma = 0\] for all 
$h \in A^p$ with $\phi(h) = 0.$  
If $F$ satisfies the above conditions, then 
\[\int_{{\mathbb{D}}} h |F|^{p-1} \conj{\sgn F}\,  d\sigma 
= \frac{\phi(h)}{\| \phi \|}\]
for all $h\in A^p.$
Furthermore, suppose that 
$\phi(f) = \int_{\mathbb{D}} f \overline{k} \,d\sigma$ for 
some $k \in H^\infty$, and that $F \in H^\infty$.
Then 
\[
\int_{{\mathbb{D}}} h |F|^{p-1} \conj{\sgn F}\,  d\sigma 
= \int_{\mathbb{D}} h \overline{k}\,  d\sigma
\]
for any function $h \in L^1$.
\end{refthm}

Ryabykh's theorem is a result for extremal problems 
in Bergman spaces that involves Hardy space regularity.  
It says that if the kernel for a 
linear functional is not only in $A^q$ but also in $H^q$, then the extremal 
function is in $H^p$ as well as $A^q$.  

\begin{ryabykhthm}\label{ryabykh_thm}
Let $1 < p < \infty$ and let $1/p + 1/q = 1.$  Suppose that 
$\phi \in (A^p)^*$ and $\phi(f) = \int_{{\mathbb{D}}} f \conj{k} \, d\sigma$ 
for some $k \in H^q$.
Then the solution $F$ to the extremal problem \eqref{norm1} belongs to 
$H^p$ and satisfies 
\begin{equation}\label{ryabkh_estimate}
\|F\|_{H^p} \le \Bigg\{ \bigg[ \max(p-1,1)
\bigg]\frac{C_p\|k\|_{H^q}}{\| k \|_{A^q}}\Bigg\}^{1/(p-1)},
\end{equation}
where $C_p$ is the constant in \eqref{A_q_isomorphism}.
\end{ryabykhthm}
Ryabykh proved that $F \in H^p$ in \cite{Ryabykh}.  
The bound \eqref{ryabkh_estimate} was 
proved in \cite{tjf1} by a variant of Ryabykh's proof.

In \cite{tjf2}, it is shown that if $p$ is an even integer, then 
for $q \le q_1 < \infty$ the extremal function 
$F \in H^{(p-1)q_1}$ if and only if the kernel $k \in H^{q_1}$. 
It is also shown that if the Taylor coefficients of $k$ 
satisfy a certain bound then 
$F \in H^\infty$, and that 
the map sending a kernel $k \in H^q$ to 
its extremal function $F\in A^p$ is a continuous map from 
$H^q \setminus \{0\}$ 
into $H^p$.   
We show that some of these results hold for any $p$ such that
 $1 < p < \infty$ and that the others hold in weaker forms. 
It is still an open problem whether the weaker results can be improved 
so that they correspond to the results from the case when 
$p$ is an even integer. 

We need the following lemma for technical reasons. 
\begin{lemma}\label{kpolybound}
If $k$ is a polynomial, then $F' \in A^r$ for some $r > 1$, and 
$F \in H^\infty$.
\end{lemma}
This follows from Corollary 2.1 in \cite{Khavinson_Stessin}. 
See page 944 of that paper for a justification of the fact that 
$F' \in A^r$.

The next lemma is a simplified version of Lemma 1.2 from \cite{tjf2}.
\begin{lemma}\label{klb}
Suppose that $1 < p_1 < \infty$ and $1<p_2,p_3 \le \infty$, 
and also that   
\[1 = \frac{1}{p_1} + \frac{1}{p_2} + \frac{1}{p_3}.\]
Let $f_1 \in H^{p_1}$, $f_2 \in H^{p_2}$, and $f_3 \in H^{p_3}$.
Suppose further that $f_1 f_2 f_3'$ is in $A^1$.
Then 
\begin{equation*}
\left| \int_{{\mathbb{D}}} \conj{f_1} f_2 f_3' \, d\sigma \right| \le C 
\|f_1\|_{H^{p_1}} \|f_2\|_{H^{p_2}} \|f_3\|_{H^{p_3}} 
\end{equation*}
where $C$ depends only on $p_1$ and $p_2$. 
Moreover, if $p_3 < \infty$ then
\begin{equation*}
\int_{{\mathbb{D}}} \conj{f_1} f_2 f_3'\, d\sigma = 
\lim_{n\rightarrow\infty} \int_{{\mathbb{D}}} \conj{f_1} f_2 (S_n f_3)'\, d\sigma.
\end{equation*}
\end{lemma}
The assumption on $f_1 f_2 f_3'$ is not essential, but without it the 
integral on the left needs to be replaced by a principle value. 
In the next lemma, the notation $\|f\|_{A^\infty}$ means the $L^\infty$ norm 
of $f$ on the disc, which of course is equivalent to the $H^\infty$ norm.
\begin{lemma}\label{lemma:iso_h_d}
If $1 \le p \le \infty$ and 
$f$ is an analytic function with derivative in $A^p$, then 
\[ \|f \|_{A^{2p}} \le \| f \|_{H^p} \le \| f' \|_{A^p} + |f(0)|. \]
The first inequality holds if $f \in H^p$. 
\end{lemma}
\begin{proof}
The first inequality in this statement is from 
\cite{Dragan_Isoperimetric}, and actually holds for 
$ 0 < p \le \infty$.  
To prove the second inequality for $1 \le p < \infty$ 
note that if $f(0)=0$ then 
\[
\| f \|_{H^p}^p = 
\frac{1}{2\pi} \int_0^{2\pi} \left| 
  \int_0^1 f'(re^{i\theta}) e^{i\theta} \,dr \right|^p d\theta 
\le 
\frac{1}{2\pi} \int_0^{2\pi} \int_0^1 |f'(re^{i\theta})|^p \, dr \, d\theta
\]
by Jensen's inequality.  But by Fubini's theorem, the last displayed 
expression equals
\[
\int_0^1 M_p^p(r,f') \, dr =
\int_0^{1/2} M_p^p(r,f') + M_p^p(1-r,f') \, dr.
\]
But the integrand in the last integral is less than or equal to 
\[2r M_p^p(r,f') + 2(1-r) M_p^p(1-r,f')\] since 
$M_p^p(r,f') \le M_p^p(1-r, f')$. But this means that the 
last displayed integral is bounded above by 
\[
\int_0^1 M_p^p(r, f') 2r \, dr = \|f'\|_{A^p}^p.
\]
If $f(0) \ne 0$ note that 
\[
\| f \|_{H^p} \le \| f - f(0) \|_{H^p} + |f(0)| \le 
\|f'\|_{A^p} + | f(0)|.
\]
The proof of the second inequality in the case $p=\infty$ is even 
easier, since then 
$|f(e^{i\theta})| \le \sup_{0\le r < 1} |f'(re^{i\theta})| + |f(0)|$ 
for each $\theta$. 
\end{proof}

\section{The Norm-Equality For Polynomials}
Let $1 < p < \infty$ and let $q$ be its conjugate exponent.
Let $k \in H^q$ and let $F$ be the extremal function in $A^p$ 
for $k$. 
We will denote by $\phi$ the functional associated with $k$. 
Define $K$ by
\begin{equation}\label{K_eq}
K(z) = \frac{1}{z} \int_0^z k(\zeta)\, d\zeta;
\end{equation}
thus  $(zK)' = k$.
Note that $\|K\|_{H^q} \le \|k\|_{H^q}$ 
(see \cite{tjf1}, equation (4.2)). 

The first result in this article corresponds to Theorem 2.1 in \cite{tjf2}. 
\begin{theorem}\label{norm_formula_ext}
Let $1<p<\infty$, let $k$ be a polynomial that is not 
identically $0$, 
and let $F\in A^p$ be the extremal function for $k.$  Then 
\begin{equation*}
\frac{1}{2\pi} \int_0^{2\pi} |F(e^{i\theta})|^p h(e^{i\theta}) \, d\theta = 
\frac{1}{2\pi \|\phi\|} \int_{0}^{2 \pi}
F \left[\left(\frac{p}{2}\right)h\conj{k} + 
\left(1-\frac{p}{2}\right) (zh)'\conj{K}\right]\,d\theta
\end{equation*}
for every polynomial $h$.
\end{theorem}
The proof of this Theorem is very similar to 
the proof of Theorem 2.1 in \cite{tjf2}.  
However, the proof in \cite{tjf2} also works if 
$k$ is any $H^q$ function.

\begin{proof}
Note that $F' \in A^{s}$ for some $s > 1$. 
By Ryabykh's theorem, $F \in H^p$.  Now, 
\begin{equation*}
\frac{1}{2\pi} \int_0^{2\pi} |F(e^{i\theta})|^p h(e^{i\theta}) \, d\theta = 
\lim_{r\rightarrow 1} \frac{i}{2\pi r^2} 
\int_{\partial (r{\mathbb{D}})} |F(z)|^p h(z)z\, d\conj{z},
\end{equation*}
where $h$ is any polynomial. 
Apply the Cauchy-Green theorem and take the limit as $r \rightarrow 1$ 
to transform the right-hand side into
\begin{equation*}
\frac{1}{\pi} \int_{{\mathbb{D}}} \left((zh)'F + \frac{p}{2}zhF'\right)|F|^{p-1}\conj{\sgn F}\,dA(z).
\end{equation*}
We may apply 
Theorem \ref{integral_extremal_condition} 
to reduce the last expression 
to  
\begin{equation}\label{eq_normeq1}
\frac{1}{\pi \|\phi\|} 
\int_{{\mathbb{D}}} \left((zh)'F + \frac{p}{2}zhF'\right) \conj{k} \,dA(z). 
\end{equation}
To prepare for a reverse application of the Cauchy-Green theorem, we rewrite 
the integral in \eqref{eq_normeq1} as
\begin{equation*}
\begin{split}
\lim_{r\rightarrow 1} \frac{1}{\pi\|\phi\|} \int_{r{\mathbb{D}}} \bigg[
\frac{\partial}{\partial \conj{z}}\left\{(zh)'F\conj{zK}\right\} &+ 
\frac{p}{2} \frac{\partial}{\partial z}\left\{zhF \conj{k}\right\} \\  
&- \frac{p}{2} \frac{\partial}{\partial \conj{z}}\left\{(zh)'F \conj{zK}\right\} \bigg] dA(z).
\\
\end{split}
\end{equation*}
Since $F$ is in $H^p$ and both $k$ and $K$ are in $H^q$, 
we may apply the Cauchy-Green theorem and 
take the limit as $r \rightarrow 1$ to see
that the above expression equals 
\begin{equation*}
\begin{split} 
&\frac{1}{2\pi i \|\phi\|} \int_{\partial {\mathbb{D}}} (zh)'F \conj{zK}\, dz + 
\frac{ip}{4\pi  \|\phi\|} \int_{\partial {\mathbb{D}}} zh F \conj{k}\, d\conj{z}\\
&\qquad \qquad - 
\frac{p}{4\pi i \|\phi\|} \int_{\partial {\mathbb{D}}} (zh)'F \conj{zK}\, dz 
\\
&= \frac{1}{2\pi\|\phi\|} 
\int_0^{2\pi} \left[ (zh)'F \conj{K} + \frac{p}{2}hF\conj{k} 
- \frac{p}{2}(zh)'F\conj{K} \right] d\theta. 
\\
\end{split}
\end{equation*}
\end{proof}

As in \cite{tjf2}, taking $h=1$ gives the following corollary, 
which we call the ``norm-equality.''
\begin{corollary}\label{norm_equality}{\rm\bf (The Norm-Equality).}
Let $1<p<\infty$, let $k$ be a polynomial that is not identically $0$, 
and let $F$ be the extremal 
function for $k.$  Then 
\begin{equation*}
\frac{1}{2\pi} \int_0^{2\pi} |F(e^{i\theta})|^p d\theta = 
\frac{1}{2\pi \|\phi\|} \int_{0}^{2 \pi}
F \left[\left(\frac{p}{2}\right)\conj{k} +\left(1-\frac{p}{2}\right) \conj{K}\right]\,d\theta.
\end{equation*}
\end{corollary}

We use the norm-equality to give the following 
theorem, which corresponds with Theorem 2.3 in \cite{tjf2}.  
Unfortunately, the theorem in this article is weaker, and it seems 
difficult to prove a statement as strong as the one in \cite{tjf2}. 
In the statement of the theorem, $F_n \rightharpoonup F$ means that 
$F_n$ converges to $F$ in the weak sense. 
\begin{theorem}\label{thm_cont}
Let $\{k_n\}$ be a sequence of functions in $H^q \setminus \{0\}$
and let 
$k_n \rightarrow k$ in $H^q$, where $k$ is 
not identically zero.
Let $F_n$ be the $A^p$ extremal function for $k_n$ and let 
$F$ be the $A^p$ extremal function for $k.$  
Then $F_n \rightharpoonup F$ in $H^p.$  Furthermore, if 
$k$ and all the $k_n$ are polynomials, then $F \rightarrow F$ in 
$H^p$.  
\end{theorem}
Because the operator taking a kernel to its extremal function is not 
linear, one cannot automatically conclude 
that $F_n \rightarrow F$ just because the operator is bounded.  
It seems likely that 
$F_n \rightarrow F$ holds for any $k_n$ and $k$ in $H^q$ such that 
$k_n \rightarrow k$, and not 
just for polynomials, but we do not know a proof of this.

\begin{proof}
The proof is basically identical to the corresponding proof in 
\cite{tjf2}, but we will summarize it for the sake of completeness.

To see that $F_n \rightharpoonup F$ in $H^p$,
note that if $F_n$ did not approach $F$ weakly in $H^p$, then 
since Ryabykh's theorem implies that the sequence 
$\{F_n\}$ is bounded in $H^p$ norm, the
Banach-Alaoglu theorem and the reflexivity of $H^p$ 
would imply that some subsequence would 
converge weakly, and thus 
pointwise, to a function not equal to $F$.  But 
$k_n \rightarrow k$ in $A^q$, and it is proved in \cite{tjf1} that 
this implies $F_n \rightarrow F$ in $A^p$, which implies  
$F_n \rightarrow F$ pointwise, a contradiction.

If $k$ and all the $k_n$ are polynomials, then 
the fact that $F_n \rightharpoonup F$ together with 
the norm-equality implies that 
$\|F_n\|_{H^p} \rightarrow \|F\|_{H^p}$. Since $H^p$ is uniformly convex, 
it follows from $F_n \rightharpoonup F$ and 
$\|F_n\|_{H^p} \rightarrow \|F\|_{H^p}$ that
$F_n \rightarrow F$ in $H^p$. 

\end{proof}

\section{Fourier Coefficients of $|F|^p$}

We now give some results about the Fourier coefficients of $|F|^p$ 
that follow from Theorem \ref{norm_formula_ext}. 
The first result gives information about the Fourier coefficients of 
$|F|^p$ for nonpositive indices.  Since $|F|^p$ is real valued, it also 
indirectly gives information about the Fourier coefficients for 
positive indices. 
\begin{theorem}\label{fourier_fp}
Let $1<p<\infty$.  Let $k$ be a polynomial (not the zero polynomial), 
let 
$F$ be the $A^p$ extremal function for $k$, and define 
$K$ by equation \eqref{K_eq}. Then for any integer $m \ge 0$, 
\begin{equation*}
\frac{1}{2\pi}\int_0^{2\pi} |F(e^{i\theta})|^p e^{im\theta} d\theta = 
\frac{1}{2\pi \|\phi\|} \int_0^{2\pi} Fe^{im\theta} 
\left[ \left(\frac{p}{2}\right) \conj{k} + \left( 1-\frac{p}{2}\right)(m+1)\conj{K}\right]
d\theta.
\end{equation*}
\end{theorem}
\begin{proof}
Take $h(e^{i\theta}) = e^{im\theta}$ in Theorem 
\ref{norm_formula_ext}. 
\end{proof}

The next result is a bound on the Fourier coefficients of $|F|^p$. 
\begin{theorem}\label{Fp_bound}
Let $1<p<\infty$. Let $k$ be a polynomial that is not the zero polynomial, 
and let $k$ have associated functional $\phi \in (A^p)^*$. 
Let $F$ be the $A^p$ extremal 
function for $k$.  
Define 
\begin{equation*}
b_m = \frac{1}{2\pi}\int_0^{2\pi}|F(e^{i\theta})|^p e^{-im\theta} d\theta,
\end{equation*}
and let 
\begin{equation*}
k(z) = \sum_{n=0}^N c_n z^n.
\end{equation*}
Then, for each $m \ge 0,$
\begin{equation*}
|b_m| = |b_{-m}| \le \frac{p}{2\|\phi\|}  \|F\|_{H^2}
\left[ \sum_{n=m}^N |c_n|^2 \right]^{1/2}.
\end{equation*}
\end{theorem}

The proof of the theorem is identical to the one found in \cite{tjf2}, 
and thus will be omitted.   
An interesting observation is that this theorem implies that 
$|F|^p$ is a trigonometric polynomial of 
degree at most $N$. 

The estimate in Theorem \ref{Fp_bound} can be used to obtain information 
about the size of $|F|^p$ (and thus of $F$), 
as in the following corollary. 
\begin{corollary}\label{order_k_linfty_condition}
If $k \in H^q \setminus \{0\}$ and 
if $c_n = O(n^{-\alpha})$ for some $\alpha > 3/2$, then $F \in H^\infty$.
\end{corollary}
\begin{proof}
Assume first that $k$ is a polynomial. 
Observe that for $m \ge 2$ we have 
\begin{equation*}
\sum_{n=m}^\infty (n^{-\alpha})^2 \le \int_{m-1}^\infty x^{-2\alpha} dx 
= \frac{(m-1)^{1-2\alpha}}{2\alpha-1}, 
\end{equation*}
and thus
\begin{equation*}
\left[ \sum_{n=m}^\infty |c_n|^2 \right]^{1/2}  \le 
C \frac{(m-1)^{(1/2)-\alpha}}{\sqrt{2\alpha-1}},
\end{equation*}
where $C$ is the constant implicit in the expression 
$O(n^{-\alpha})$. 
Thus we have (for $m \ge 2$) that 
\[
|b_m| = |b_{-m}| \le C \frac{p}{2\|\phi\|}  \|F\|_{H^2}
\frac{(m-1)^{(1/2)-\alpha}}{\sqrt{2\alpha-1}} 
\]
Therefore,
\[
\begin{split}
\sum_{m=3}^{\infty} |b_{-m}| = 
\sum_{m = 3}^{\infty} |b_m| 
&\le
C \frac{p}{2\|\phi\|}  \|F\|_{H^2}
 \int_2^\infty \frac{(x-1)^{(1/2)-\alpha}}{\sqrt{2\alpha-1}} \, dx
\\ &\le
C \frac{p}{2\|\phi\|}  \|F\|_{H^2}
\frac{1}{(\alpha - 3/2)\sqrt{2\alpha-1}} .
\end{split}
\]
But this implies that 
\[
\|F\|_{H^\infty}^p = \| |F|^p \|_{L^\infty} \le 
\sum_{n=-\infty}^\infty |b_{m}| \le 
C
\frac{p}{\|\phi\|}  \|F\|_{H^2}
\frac{1}{(\alpha - 3/2)\sqrt{2\alpha-1}} 
 + \sum_{m=-2}^{2} |b_m|.
\]
Since each $|b_m| \le \|F\|_{H^p}^p < \infty$, 
we have that 
\[
\|F\|_{H^\infty}^p \le 
C
\frac{p}{\|\phi\|}  \|F\|_{H^2}
\frac{1}{(\alpha - 3/2)\sqrt{2\alpha-1}} 
 + 5 \|F\|_{H^p}^p.
\]
Since $\|F\|_{H^2} \le \|F\|_{H^\infty} < \infty$, we have that 
\[
\|F\|_{H^\infty}^{p-1} \le 
C \frac{p}{\|\phi\|} 
\frac{1}{(\alpha - 3/2)\sqrt{2\alpha-1}} 
 + 5 \| F \|_{H^p}^p.
\]
Here we have also used the fact that 
$\|F\|_{H^\infty}^{-1} \le \|F\|_{A^p}^{-1} = 1$. 

Now we drop the assumption that 
$k$ is a polynomial.  
Let $F_n$ be the extremal function for 
$S_n k$, and let $\phi_n$ be the corresponding functional. 
By Ryabykh's theorem 
and the fact that $S_n k \rightarrow k$ in $H^q$, 
the sequence $\|F_n\|_{H^p}$ is bounded.  
Now, the above displayed inequality holds with $F_n$ in place of $F$ 
and $\phi_n$ in place of $\phi$,   
since $C$ can be taken to be independent of $n$. 
Also, it follows from the fact that 
$S_n k \rightarrow k$ in $A^q$ that 
$\phi_n \rightarrow \phi$ in $(A^p)^*$, and that
$F_n \rightarrow F$ in $A^p$ and thus 
uniformly on compact subsets.  
Therefore,
\[
\|F\|_{H^\infty}^{p-1} 
\le \liminf_{n \rightarrow \infty} \|F_n\|_{H^\infty}^{p-1} 
\le C \frac{p}{\|\phi\|} 
\frac{1}{(\alpha - 3/2)\sqrt{2\alpha-1}} 
 + 5 \liminf_{n \rightarrow \infty} \| F_n \|_{H^p}^p.
\]
This proves the result.
\end{proof}

\section{Relations Between the Size of the Kernel and Extremal Function}

In this section we show that if $1<p<\infty$ and 
$q \le q_1 < \infty$ and the kernel $k \in H^{q_1}$ then the 
extremal function $F \in H^{(p-1)q_1}$.
For $q_1=q$ the statement reduces to Ryabykh's theorem.  
For $p$ an even integer, this statement and its converse are proved 
in \cite{tjf2}. 
It is still an open problem to decide if the converse holds 
for general $p$, although we prove a 
weaker result similar to it. 

We first prove the following theorem. 
\begin{theorem}\label{extremal_bound}
Let $1<p<\infty$ and 
let $q = p/(p-1)$ be its conjugate exponent. 
Let $F \in A^p$ be the extremal function corresponding to the kernel 
$k \in A^q$, where $k$ is a polynomial.
Let $p \le p_1 < \infty$, and  
$q \le q_1 < \infty$. 
Define $p_2$ by 
\[
\frac{1}{q_1} + \frac{1}{p_1} + \frac{1}{p_2} = 1.\]
Then for every trigonometric polynomial $h$ we have  
\begin{equation*}
\left| \int_0^{2\pi} |F(e^{i \theta})|^p h(e^{i\theta})\, d\theta \right| \le 
C\frac{\|k\|_{H^{q_1}}}{\|k\|_{A^q}} \|F\|_{H^{p_1}} 
\|h\|_{L^{p_2}},
\end{equation*}
where $C$ is some constant depending only on $p$, $p_1$, and $q_1$.
\end{theorem}
Note that the case $p_2 = \infty$ occurs if and only if 
$q=q_1$ and $p=p_1$.  The theorem is then a trivial consequence of Ryabykh's 
theorem, so we need only prove the theorem if $p_2 < \infty$. 
\begin{proof}
Let $h$ be an analytic polynomial.
In the proof of Theorem \ref{norm_formula_ext}, we showed that 
\[
\begin{split}
 \frac{1}{2\pi} \int_0^{2\pi} |F(e^{i\theta})|^p h(e^{i\theta})\, d\theta = 
\frac{1}{\pi\|\phi\|} 
\int_{{\mathbb{D}}} \left((zh)'F + \frac{p}{2}zhF'\right) \conj{k} \,dA(z).\\
\end{split}
\]
Apply Lemma \ref{klb} separately to the two parts of the integral to 
conclude that its absolute value is bounded by
\begin{equation*}
C \frac{1}{\|\phi\|}\|k\|_{H^{q_1}} \|f\|_{H^{p_1}} \|h\|_{H^{p_2}},
\end{equation*}
where $C$ is a constant depending only on $p_1$ and $q_1$. 
Since 
\[\frac{1}{\|\phi\|} \le \frac{C_p}{\|k\|_{A^q}}\] 
by equation \eqref{A_q_isomorphism},
the desired result holds for the case where 
$h$ is an analytic polynomial.
If $h$ is an arbitrary trigonometric polynomial, then as in 
\cite{tjf2} the boundedness of the Szeg\H{o} projection can be used to 
show the result holds.
\end{proof}

For a given $q_1 > q$, we will apply the theorem just proven 
with $p_1 = (p-1) q_1$ and 
with $p_2'$ chosen 
to equal $p_1/p$, where $p_2'$ is the conjugate exponent to $p_2$.
This allows us to bound the $H^{p_1}$ norm of $f$ in terms of 
$\|\phi\|$ and $\|k\|_{H^{q_1}}$ only. 
\begin{theorem}\label{r_ext}
 Let $1 < p < \infty$, and let $q$ be its conjugate 
exponent. Let $F\in A^p$ be the extremal function for a kernel 
$k \in A^q$. 
If for some $q_1$ such that $q \le q_1 < \infty$ the kernel  
$k \in H^{q_1},$ then $F \in H^{p_1}$ for $p_1=(p-1)q_1.$  In fact, 
\begin{equation*}
\|F\|_{H^{p_1}} \le C\left(
\frac{\|k\|_{H^{q_1}}}{\|k\|_{A^q}}\right)^{1/(p-1)},
\end{equation*}
where $C$ depends only on $p$ and $q_1$.
\end{theorem}
The proof of this theorem is identical to the proof of the corresponding
theorem in \cite{tjf2}, so we give a summary.
\begin{proof}
The case $q_1 = q$ is Ryabykh's theorem, so we assume $q_1 > q.$ 
Let $p_1 = (p-1)q_1$; thus $p_1 > p = (p-1)q.$  
Let $p_2 = p_1/(p_1-p)$, 
so \[
\frac{1}{q_1} + \frac{1}{p_1} + \frac{1}{p_2} = 1\]  
and
$p_2' = p_1/p$ and $1<p_2<\infty$. 
Let $F_n$ denote the extremal function corresponding to the kernel 
$S_n k$ 
(where we choose $n$ large enough so that $S_nk$ is not identically zero).
Then 
for any trigonometric polynomial $h$, Theorem \ref{extremal_bound} implies 
that  
\begin{equation*}
\left|\frac{1}{2\pi}\int_0^{2\pi} |F_n|^p h(e^{i\theta}) d\theta \right|\le 
C \frac{\|S_n k\|_{H^{q_1}}}{\|S_n k\|_{A^q}} 
\|F_n\|_{H^{p_1}} \|h\|_{L^{p_2}}. 
\end{equation*}
Taking the supremum over all trigonometric polynomials $h$ 
with $\|h\|_{L^{p_2}} \le 1$ gives 
\begin{equation*}
\|F_n\|_{H^{p_1}}^p =  \| |F_n|^p\|_{L^{p_2'}} \le
C \frac{\|S_n k\|_{H^{q_1}}}{\|S_n k\|_{A^q}} \|F_n\|_{H^{p_1}}.
\end{equation*}
Because $\|F_n\|_{H^{p_1}} < \infty$ (since $S_n k$ is a polynomial) 
we may divide both sides of the inequality by $\|F_n\|_{H^{p_1}}$ to obtain
\begin{equation*}
\|F_n\|_{H^{p_1}}^{p-1} \le C  \frac{\|S_n k\|_{H^{q_1}}}{\|S_n k\|_{A^q}},
\end{equation*}
where $C$ depends only  on $p$ and $q_1$. Taking the limit as 
$n \rightarrow \infty$ gives the desired result.
\end{proof} 

Recall from Section \ref{intro} that if
$F \in A^p$ has unit norm, there is a 
corresponding kernel $k \in A^q$ such that $F$ is the extremal function for 
$k$, and that this kernel is uniquely determined up to a positive multiple.
Thus, it makes sense to ask if the converse of 
Theorem \ref{r_ext} holds.  That is, does $F \in H^{(p-1)q_1}$ imply
that  
$k \in H^{q_1}$?  If $p$ is an even integer and $q \le q_1 < \infty$ then 
by Theorem 4.3 in \cite{tjf2} this is the case.  
In fact, the proof in \cite{tjf2} works for 
any $q_1$ such that $1 < q_1 < \infty$ (as long as $p$ is an even integer). 
For general $p$ we do not know if the result is still true.
The result does hold if $2 \le p < \infty$ and $1 < q_1 < \infty$ and if 
$F$ is nonvanishing, since the proof in \cite{tjf2} works in that 
case.  For general $F$
we can prove the following weaker result for $2 \le p < \infty$. 

\begin{theorem} Suppose $2 \le p < \infty$ and $1 < q_1 < \infty$. 
Let $F \in A^p$ with $\|F\|_{A^p}=1$, 
and let $k$ be a kernel such that $F$ is the extremal function for $k$. 
Let $p_1 = q_1(p-1)$ and let $p_2 = p q_1 / (q_1 + 1)$. 
If $F \in H^{p_1}$ and $F' \in A^{p_2}$
then   
$k \in H^{q_1}$ and
\begin{equation*}
\frac{\|k\|_{H^{q_1}}}{\|k\|_{A^q}} \le
 \left[\csc\left(\frac{\pi}{p}\right)\right]
\left( \|F\|_{H^{p_1}}^{p-1} + 
     \frac{p-2}{2} (\| F' \|_{A^{p_2}} + |F(0)|)^{p-2} 
       \| F' \|_{A^{p_2}}\right),
\end{equation*} 
where $C_p$ is as in inequality \eqref{A_q_isomorphism}. 
\end{theorem}

\begin{proof}
Note first that the case $p=2$ is trivial since then $F$ and $k$ are 
constant multiples of each other, so assume $p \ne 2$. 
Let $q$ denote the exponent conjugate to $p$. 
Let $h$ be a polynomial and let $\phi$ be the functional in $(A^p)^*$ 
corresponding to $k$. Then by Theorem \ref{integral_extremal_condition}
and the Cauchy-Green theorem, 
\begin{equation*}
\begin{split}
&\quad \frac{1}{\|\phi\|} \int_{{\mathbb{D}}} \conj{k(z)} (zh(z))' d \sigma \\ &= 
   \int_{{\mathbb{D}}} |F(z)|^{p-1}\sgn(\conj{F(z)}) (zh(z))' d \sigma \\
 &=  \lim_{r \rightarrow 1}
\int_{r\mathbb{D}} \left\{ \frac{\partial}{\partial z} \left[
           |F|^{p-1} \sgn \conj{F} zh \right] - 
   \frac{p-2}{2} |F|^{p-2} F' \sgn \conj{F}^2 zh \right\} \frac{dA}{\pi} \\
& = \lim_{r \rightarrow 1}
  \frac{i}{2 \pi } \int_{\partial(r \mathbb{D})} 
          |F|^{p-1} \sgn \conj{F} zh \, d \conj{z} - 
   \int_{\mathbb{D}}  \frac{p-2}{2} |F|^{p-2} F' \sgn \conj{F}^2 zh \, d\sigma \\
& = \frac{1}{2 \pi } \int_0^{2\pi} 
          |F|^{p-1} (\sgn \conj{F}) h \, d \theta - 
     \int_{\mathbb{D}}  \frac{p-2}{2} |F|^{p-2} F' \sgn \conj{F}^2 zh \, d\sigma. 
\end{split}
\end{equation*}
Here we have used the fact that $|F|^{p-2} F' \in L^1$, which 
follows from the fact that
$(p-2)/p + 1/p_2 < 1$. 
Now apply H{\"o}lder's inequality to the first integral using 
exponents $q_1$ and $q_1' = q_1 / (q_1 - 1)$, and apply 
it to the second using exponents $2p_2/(p-2)$ and  
$p_2$ and $2q_1'$ 
to obtain that the above expression is bounded above in absolute value 
by 
\[
\|F\|_{H^{p_1}}^{p-1} \|h\|_{H^{q_1'}} + 
     \frac{p-2}{2} \| F \|_{A^{2p_2}}^{p-2} 
       \| F' \|_{A^{p_2}} 
            \|h\|_{A^{2q_1'}}.
\]
But by Lemma \ref{lemma:iso_h_d}, this is at most
\[
\left( \|F\|_{H^{p_1}}^{p-1} + 
     \frac{p-2}{2} (\| F' \|_{A^{p_2}} + |F(0)|)^{p-2} 
       \| F' \|_{A^{p_2}}\right) 
            \|h\|_{H^{q_1'}}.
\]
Let $C$ equal the part of the above expression in parentheses. 
Then 
\[ \left|\int_{{\mathbb{D}}} \conj{k(z)} (zh(z))' d\sigma \right|
 \le 
C \|\phi \| \|h\|_{H_{q_1'}}\] for all polynomials $h$, and we 
can define a continuous linear functional $\psi$ on $H^{q_1'}$ so that 
\[
\psi(h) = \int_{{\mathbb{D}}} \conj{k(z)} (zh(z))' d\sigma
\]
for all polynomials $h$.  
Then $\psi$ has an associated kernel in $H^{q_1}$ (see p.\ 113 of 
\cite{D_Hp}).  Call the kernel $\widetilde{k}.$ 
For $h \in H^{q_1'}$ it follows that 
\[ \psi(h) = \frac{1}{2\pi}\int_0^{2\pi} \conj{\widetilde{k}(e^{i\theta})} 
h(e^{i\theta}) \, d\theta. \]
By the Cauchy-Green theorem, 
\begin{equation}\label{kernel_trick_eqn}
\begin{split}
 \int_{{\mathbb{D}}} \conj{k(z)} (zh(z))' \,d\sigma &= 
\psi(h) 
\\&= \frac{1}{2\pi}\int_0^{2\pi} 
    \conj{\widetilde{k}(e^{i\theta})} h(e^{i\theta}) \,d\theta 
\\ &
=\lim_{r\rightarrow 1} \frac{i}{2\pi}
\int_{\partial (r{\mathbb{D}})} \conj{\widetilde{k}(z)} h(z)z \,d\conj{z} \\
&=
\lim_{r\rightarrow 1} 
\int_{r{\mathbb{D}}} \conj{\widetilde{k}(z)} (zh(z))' \frac{dA}{\pi} \\&= 
\int_{{\mathbb{D}}} \conj{\widetilde{k}(z)} (zh(z))' \,d\sigma,
\end{split}
\end{equation}
where $h$ is any polynomial. 

Define the polynomial $H$ by 
\begin{equation*} H(z) = \frac{1}{z} \int_{0}^z h(\zeta) \, d\zeta.
\end{equation*}
Then substituting $H(z)$ for $h(z)$ in equation \eqref{kernel_trick_eqn}, 
and using the fact that 
$(zH)' = h$, we have 
\[\int_{{\mathbb{D}}} \conj{\widetilde{k}(z)} h(z) \, d\sigma = 
 \int_{{\mathbb{D}}} \conj{k(z)} h(z) \, d\sigma\]
for every polynomial $h$. 
Since $k \in A^q$ and 
$\widetilde{k} \in H^{q_1} \subset A^{2q_1}$, we have 
that the power series for $k$ and $\widetilde{k}$ converge in 
$A^q$ and $A^{2q_1}$ respectively.
Using this fact and choosing $h(z) = z^n$ for $n \in \mathbb{N}$ 
shows that the power series of $k$ and $\widetilde{k}$ are identical, 
and so $k = \widetilde{k}$ and $k \in H^{q_1}$. 

Now for any polynomial $h$,
\begin{equation*}
\left|
\frac{1}{2\pi}\int_0^{2\pi} \conj{k(e^{i\theta})} h(e^{i\theta}) d\theta 
\right |
\le
C\ \|\phi \|\|h\|_{H^{q_1'}} 
\le 
C\ \|k\|_{A^q} \|h\|_{H^{q_1'}}.
\end{equation*}
where we have used inequality \eqref{A_q_isomorphism}.
But for any trigonometric polynomial $h$, we have 
\begin{equation*}
\begin{split}
\left|
\frac{1}{2\pi}\int_0^{2\pi} \conj{k(e^{i\theta})} h(\theta) \, d\theta \right| 
&=
\left|
\frac{1}{2\pi}\int_0^{2\pi} \conj{k(e^{i\theta})} \left[S(h)(e^{i\theta})\right] d\theta  \right|
\\ &\le
C\|k\|_{A^q}  \|S(h)\|_{H^{q_1'}} \\&\le
C \csc\left(\frac{\pi}{p}\right)  \|k\|_{A^q} \|h\|_{L^{q_1'}},
\end{split}
\end{equation*}
where $S$ denotes the Szeg\H{o} projection.  Note that  
$\csc(\pi/p)$ is the norm of the Szeg\H{o} projection on 
$L^p(\partial \mathbb{D})$ 
(see \cite{Hollenbeck_Verbitsky-Riesz}).
Now take the supremum over all trigonometric polynomials $h$ with 
$\|h\|_{L^{q_1'}} \le 1$ and 
divide both sides of the inequality by $\|k\|_{A^q}$. 
\end{proof}
It is interesting to note that the value of $p_2$ 
in the above theorem is less than $p$ no matter the value of $q_1$. 

\section{Open Problems and a Simple Result}

As we have noted, 
unlike in the case in which $p$ is an even integer, we do not know how 
to show that if $F \in H^{(p-1)q_1}$ then $k \in H^{q_1}$.  
However, we can show that a corresponding result holds if we replace the 
Hardy spaces by Bergman spaces.  
This result is not difficult and may be well known, but we do not know of 
anywhere it appears in the literature. 
\begin{theorem}
Let $1 < p < \infty$.
Suppose $k \in A^q$ and $F$ is the $A^p$ extremal function for $k$. 
If $F \in A^{(p-1)q_1}$ for $1 <  q_1 < \infty$, then $k \in A^{q_1}$. 
If $F \in H^\infty$, then $k$ is in the Bloch space, and if $F$ is 
continuous on the closed disc, then $k$ is in the little Bloch space.
\end{theorem}
\begin{proof}
As stated about, $k$ must be a positive scalar multiple of 
$\mathcal{P}(|F|^{p}/\overline{F}) = \mathcal{P}(|F|^{p-1} \sgn F)$, 
where $\mathcal{P}$ is the Bergman projection. 
The result now follows since the Bergman projection is bounded from 
$L^r$ to $A^r$ for $1 < r < \infty$, and since it maps $L^\infty$ onto the 
Bloch space and the space of continuous functions on the closed disc 
onto the little Bloch space (see e.g.\ \cite{D_Ap}). 
\end{proof}

We now mention some open problems that could motivate further study. 
\begin{enumerate}
\item For $1 < p < \infty$, if $F \in H^{(p-1)q_1}$, is $k \in H^{q_1}$? 
As we have said, this is known from \cite{tjf2} 
to be true if $p$ is an even integer, or if $F$ is nonvanishing and 
$2 \le p < \infty$.
\item Is it the case that if $k \in A^{q_1}$, where $1 < q_1 < \infty$, then 
$F$ must be in $A^{(p-1)q_1}$?  If not, can anything interesting be said about 
the regularity of $F$?
\item If $k$ is in the Bloch space or the little Bloch space, 
can anything 
of interest be said about the regularity of $F$?
\item If $k \in H^\infty$, must $F \in \mathrm{BMO}$?  
If $F \in H^\infty$, must $k \in \mathrm{BMO}$?  
\item 
Does the generalization of Ryabykh's theorem 
(Theorem \ref{r_ext}) hold for $1 < q_1 < q$?
\item 
Is the mapping from kernels to Bergman space extremal functions 
continuous on Hardy spaces? 
Is the mapping from extremal functions to kernels continuous 
on Hardy spaces?  
(Of course, there are multiple kernels with the same extremal 
function, but they are all positive scalar multiples of each other, so 
one can make sense of this question by specifying which kernel is chosen). 
\end{enumerate}

\providecommand{\bysame}{\leavevmode\hbox to3em{\hrulefill}\thinspace}
\providecommand{\MR}{\relax\ifhmode\unskip\space\fi MR }
\providecommand{\MRhref}[2]{%
  \href{http://www.ams.org/mathscinet-getitem?mr=#1}{#2}
}
\providecommand{\href}[2]{#2}


\begin{thebibliography}{10}

\bibitem{Beneteau_Khavinson_survey}
Catherine B{\'e}n{\'e}teau and Dmitry Khavinson, \emph{A survey of linear
  extremal problems in analytic function spaces}, Complex analysis and
  potential theory, CRM Proc. Lecture Notes, vol.~55, Amer. Math. Soc.,
  Providence, RI, 2012, pp.~33--46. \MR{2986891}

\bibitem{Dostanic_BP_Norm}
Milutin~R. Dostani{\'c}, \emph{Two sided norm estimate of the {B}ergman
  projection on {$L^p$} spaces}, Czechoslovak Math. J. \textbf{58(133)} (2008),
  no.~2, 569--575. \MR{2411110 (2009f:32003)}

\bibitem{D_Hp}
Peter Duren, \emph{Theory of {$H\sp{p}$} spaces}, Pure and Applied Mathematics,
  Vol. 38, Academic Press, New York, 1970. \MR{0268655 (42 \#3552)}

\bibitem{DKSS_Pac}
Peter Duren, Dmitry Khavinson, Harold~S. Shapiro, and Carl Sundberg,
  \emph{Contractive zero-divisors in {B}ergman spaces}, Pacific J. Math.
  \textbf{157} (1993), no.~1, 37--56. \MR{1197044 (94c:30048)}

\bibitem{D_Ap}
Peter Duren and Alexander Schuster, \emph{Bergman spaces}, Mathematical Surveys
  and Monographs, vol. 100, American Mathematical Society, Providence, RI,
  2004. \MR{2033762 (2005c:30053)}

\bibitem{tjf1}
Timothy Ferguson, \emph{Continuity of extremal elements in uniformly convex
  spaces}, Proc. Amer. Math. Soc. \textbf{137} (2009), no.~8, 2645--2653.

\bibitem{tjf2}
Timothy Ferguson, \emph{Extremal problems in {B}ergman spaces and an extension
  of {R}yabykh's theorem}, Illinois J. Math. \textbf{55} (2011), no.~2,
  555--573 (2012). \MR{3020696}

\bibitem{tjf3}
\bysame, \emph{Solution of extremal problems in {B}ergman spaces using the
  {B}ergman projection}, Comput. Methods Funct. Theory \textbf{14} (2014),
  no.~1, 35--61. \MR{3194312}

\bibitem{Hedenmalm_canonical_A2}
H{\aa}kan Hedenmalm, \emph{A factorization theorem for square area-integrable
  analytic functions}, J. Reine Angew. Math. \textbf{422} (1991), 45--68.
  \MR{1133317 (93c:30053)}

\bibitem{Hollenbeck_Verbitsky-Riesz}
Brian Hollenbeck and Igor~E. Verbitsky, \emph{Best constants for the {R}iesz
  projection}, J. Funct. Anal. \textbf{175} (2000), no.~2, 370--392.
  \MR{1780482 (2001i:42010)}

\bibitem{Khavinson_McCarthy_Shapiro}
Dmitry Khavinson, John~E. McCarthy, and Harold~S. Shapiro, \emph{Best
  approximation in the mean by analytic and harmonic functions}, Indiana Univ.
  Math. J. \textbf{49} (2000), no.~4, 1481--1513. \MR{1836538 (2002b:41023)}

\bibitem{Khavinson_Stessin}
Dmitry Khavinson and Michael Stessin, \emph{Certain linear extremal problems in
  {B}ergman spaces of analytic functions}, Indiana Univ. Math. J. \textbf{46}
  (1997), no.~3, 933--974. \MR{1488342 (99k:30080)}

\bibitem{Ryabykh_certain_extp}
V.~G. Ryabych, \emph{Certain extremal problems}, Nauchnye Soobscheniya R.G.U.
  (1965), 33--34 ((in Russian)).

\bibitem{Ryabykh}
V.~G. Ryabykh, \emph{Extremal problems for summable analytic functions},
  Sibirsk. Mat. Zh. \textbf{27} (1986), no.~3, 212--217, 226 ((in Russian)).
  \MR{853902 (87j:30058)}

\bibitem{Shapiro_Approx}
Harold~S. Shapiro, \emph{Topics in approximation theory}, Springer-Verlag,
  Berlin, 1971, With appendices by Jan Boman and Torbj\"orn Hedberg, Lecture
  Notes in Math., Vol. 187. \MR{0437981 (55 \#10902)}

\bibitem{Dragan}
Dragan Vukoti{\'c}, \emph{Linear extremal problems for {B}ergman spaces},
  Exposition. Math. \textbf{14} (1996), no.~4, 313--352. \MR{1418027
  (97m:46117)}

\bibitem{Dragan_Isoperimetric}
\bysame, \emph{The isoperimetric inequality and a theorem of {H}ardy and
  {L}ittlewood}, Amer. Math. Monthly \textbf{110} (2003), no.~6, 532--536.
  \MR{1984405}

\end{thebibliography}
\end{document}